\documentclass[11pt, reqno]{amsart}
\usepackage{amscd}
\usepackage{amsmath}
\usepackage{amssymb}
\usepackage{latexsym}
\usepackage{amsthm}
\usepackage{graphicx}
\usepackage{tikz}
\usepackage{tikz-cd}

\usepackage[all]{xy}       

    \SelectTips{cm}{10}     

    \everyxy={<2.5em,0em>:} 

    \xyoption{web}          

\RequirePackage{color}
\definecolor{myred}{rgb}{0.75,0,0}
\definecolor{mygreen}{rgb}{0,0.5,0}
\definecolor{myblue}{rgb}{0,0,0.65}

\RequirePackage{ifpdf}
\ifpdf
  \IfFileExists{pdfsync.sty}{\RequirePackage{pdfsync}}{}
  \RequirePackage[pdftex,
   colorlinks = true,
   urlcolor = myblue, 
   citecolor = mygreen, 
   linkcolor = myred, 
   pagebackref,
   bookmarksopen=true]{hyperref}
\else
  \RequirePackage[hypertex]{hyperref}
\fi

\definecolor{darkred}{HTML}{993333}

\newcommand{\arxiv}[1]{\href{http://arxiv.org/abs/#1}{\tt arXiv:\nolinkurl{#1}}}
\newcommand{\arXiv}[1]{\href{http://arxiv.org/abs/#1}{\tt arXiv:\nolinkurl{#1}}}

\newtheorem{theorem}{Theorem}[section]
\newtheorem{lemma}[theorem]{Lemma}

\newtheorem{definition}[theorem]{Definition}
\newtheorem{example}[theorem]{Example}

\newtheorem{proposition}[theorem]{Proposition}
\newtheorem{corollary}[theorem]{Corollary}
\newtheorem{conjecture}[theorem]{Conjecture}

\theoremstyle{remark}
\newtheorem{remark}[theorem]{Remark}

\numberwithin{equation}{section}

\newcommand{\nc}{\newcommand}

\nc{\flags}{\mathcal{F}}

\def\wt{\text{wt}}
\newcommand{\spann}{\operatorname{span}}

\nc{\KP}{\operatorname{KP}}

\def\ii{{\bf i}}
\def\aa{{\bf a}}
\nc{\re}{re}

\def\im{\operatorname{im}}

\def\N{\mathbb{N}}

\def\C{\mathcal{C}}
\def\R{\mathbb{R}}
\def\Z{\mathbb{Z}}
\def\F{\mathcal{F}}

\def\Th{\Theta}
\def\one{1}

\def\E{\mathcal{E}}

\def\A{\mathcal{A}}

\def\B{\mathcal{B}}

\def\jj{{\bf j}}

\nc{\co}{\overline{\nabla}}

\def\a{\alpha}
\def\b{\beta}

\def\la{\lambda}

\def\ga{\gamma}

\def\d{\delta}

\def\th{\theta}
\def\w{\omega}

\def\f{\mathbf{f}}

\def\g{\mathfrak{g}}

\def\D{\Delta}
\def\U{\mathcal{U}}
\def\udot{\dot{\mathcal{U}}}

\def\Hom{\operatorname{Hom}}
\def\HOM{\operatorname{HOM}}

\def\Ind{\operatorname{Ind}}

\def\End{\operatorname{End}}

\def\Rep{\,\mbox{Rep}\,}
\def\Res{\operatorname{Res}}

\def\Ext{\operatorname{Ext}}

\def\mods{\mbox{-mod}}
\def\fmod{\mbox{-fmod}}
\def\prmod{\mbox{-pmod}}
\def\id{\mbox{id}}

\newcommand{\map}[2]{\,{:}\,#1\!\longrightarrow\!#2}

\def\h{\mathfrak{h}} 

\def\T{\mathcal{T}}
\def\TT{\mathcal{T}}

\def\inv{^{-1}}

\numberwithin{equation}{section}

\title[Braid Group Action]{On a Braid Group Action}
\address{}\email{maths@petermc.net}
\author{Peter J McNamara}
\date{\today}

\begin{document}

%

\

\

\begin{abstract}
We discuss some consequences of the invertibility of Rickard complexes in a categorified quantum group. Results include a description of reflection functors for quiver Hecke algebras and a theory of restricting categorical representations along a face.
\end{abstract}
\maketitle

\section{Introduction}

It is a classical construction to associate to any symmetrisable Cartan matrix a quantum group $U_q(\g)$. We concern ourselves with a categorified version of this construction, where a strict 2-category $\U$ is produced whose Grothendieck group is canonically identified with the idempotented form of the corresponding quantum group.

This theory of a categorical quantum group originates with the work of Chuang and Rouquier \cite{chuangrouquier} who used the notion of a categorical $\mathfrak{sl}_2$ action to construct interesting derived equivalences. The generalisation to general $\g$ came in the work of Khovanov and Lauda
\cite{kl3} and Rouquier \cite{rouquier}. They give two different presentations of a 2-category $\U$, the equivalence of which was shown by Brundan \cite{brundan}.

The braid group acts by algebra automorphisms on the quantum group $U_q(\g)$. The starting point of this paper is the theorem that this braid group action can be lifted to a braid group action on the homotopy category $K(\U)$ by autoequivalences $\T_s$. 
The autoequivalence $\T_s$ is given by conjugation by a Rickard complex, which is proved to be invertible in $K(\U)$ in \cite{vera}.
In this paper we explore the implications of the existence of this action. The applications which we study all require the theory of standard modules for quiver Hecke algebras, also known as KLR algebras (after Khovanov, Lauda and Rouquier). This theory was developed in \cite{bkm} in finite type and in \cite{mcn3} in affine type over a field of characteristic zero. The necessary facts from this theory are all recalled when they are needed. We prove that these standard modules are compatible with these autoequivalences $\T_s$ in a precise way in Proposition \ref{tofstandard}.

The main application studied in this paper is to the construction of reflection functors for quiver Hecke algebras.
These are functors which categorify the Satio reflection on the crystal $B(\infty)$, as well as Lusztig's braiding automorphism $T_s$, restricted to the positive part of the quantum group.
The $T_s$ do not preserve the positive part $U_q(\g)^+$, but do induce an isomorphism between two subspaces $\ker(_sr)\cong \ker(r_s)$ (see \S \ref{prelim} for the definitions of these objects).
 Both subspaces $\ker(_sr)$ and $\ker(r_s)$ are categorified by a Serre subcategory of the category of quiver Hecke modules, which we denote by $_s\C$ and $\C_s$ respectively. We show how the autoequivalence $\T_s$ induces an equivalence of the abelian categories $_s\C$ and $\C_s$.

This equivalence was obtained geometrically in finite simply laced type over a field of characteristic zero in 
\cite{kato}. 
This was subsequently generalised to finite simply laced type in all characteristics in \cite{geometry}.
The related work of \cite{xiaozhao,zhao} provides a geometric incarnation of $T_s$ in all types but does not produce results of the same strength that we provide here.

The second application is to a theory of restricting a categorical representation.
In \cite{facefunctors}, the notion of a face of a root system is introduced. Each face defines a root system and hence there is a quiver Hecke algebra associated to that face, which we will call $R_F$. The main result of \cite{facefunctors} is the construction of a fully faithful functor from $R_F\mods$ to $R\mods$, whose essential image is explicitly determined in terms of a subcategory of cuspidal modules if the face root system is of finite type. In this paper we go beyond the results of \cite{facefunctors} at the cost of passing to the homotopy category. For any face, we show how to restrict a categorical action on a category $\C$ to a categorical action for the face quantum group on the homotopy category $K(\C)$. In certain cases the restricted categorical action is actually an action on $\C$ and we give a criterion for checking this.

In this paper we work under the assumption that our Cartan data are simply laced.
The reason for restricting to simply-laced type is that we need some computational results in rank 2 which have only been carried out under this restriction, in particular Theorem \ref{aller}. We expect that generalising these computational results of \cite{allr}, to all types will allow the results of this paper to similarly become available in greater generality, subject still to the requirements that the necessary theory of standard modules for the relevant quiver Hecke algebras exists.

We thank the authors of \cite{allr} fo providing a draft of their paper, and O. Yacobi for bringing \cite{vera} to the author's attention.

\section{Preliminaries}\label{prelim}

Let $S$ be a set and $A=(a_{st})_{s,t\in S}$ be a simply-laced Cartan matrix. This means that $a_{ss}=2$ and if $s\neq t$, $a_{st}=a_{ts}\in\{0,-1\}$. Choose also a realisation of $A$. This is the additional data of a complex vector space $\h$, a set of linearly independent vectors $\a_s\in \h^*$ and a set of linearly independent vectors $\a_s^\vee\in \h$ such that $\langle \a_s^\vee,\a_t\rangle = a_{st}$. Define the weight lattice
\[
 P=\{\la\in\h^*\mid \langle \a_s^\vee,\la\rangle \in\Z \mbox{ for all } s\in S\}.
\]

To this data, there is associated a Kac-Moody Lie algebra quantum group $U_q(\mathfrak{g})$. We are most interested in its positive part $U_q^+(\mathfrak{g})$.

The positive part of the quantum group $U_q^+(\g)$ is the unital associative algebra generated by elements $\th_s$ for $s\in S$ subject to the quantum Serre relations
\begin{align*}
\th_s\th_t&=\th_s\th_t &\mbox{ if } a_{st} &=0 \\
(q+q\inv)\th_s\th_t\th_s &= \th_s\th_t^2+\th_t\th_s^2 &\mbox{ if } a_{st} &= -1
\end{align*}
The algebra $U_q^+(\g)$ has an integral form over $\Z[q,q\inv]$ which we call $\f$. It is the $\Z[q,q\inv]$-subalgebra generated by the divided powers $\th_s^n/[n]!$, where $[n]!=\prod_{i=1}^n\frac{q^i-q^{-i}}{q-q\inv}$. This algebra is graded by $\N S$ where $\th_s$ has degree $s$. 

Let $s\in S$. Let $r_s\map{\f}{\f}$ be the linear map defined inductively by
\begin{align*}
r_s(\th_t)&=\delta_{st} \\
r_s(xy)&=q^{\deg(y)\cdot s} r_s(x)y+x r_s(y).
\end{align*}
and let ${_s}r \map{\f}{\f}$ be the linear map inductively defined by
\begin{align*}
{_s}r(\th_t)&=\delta_{st} \\
{_s}r(xy)&= {_s}r(x)y+q^{\deg(x)\cdot s} x {_s}r(y).
\end{align*}
We will need $r_s$ and ${_s}r$ in \S \ref{sec:reflection}.

%

Given two objects $X$ and $Y$ in a graded category, we write $qX$ for the grading shift of $X$, $\Hom(X,Y)$ for the degree zero morphisms from $X$ to $Y$ and $$\HOM(X,Y)=\bigoplus_{d\in \Z}\Hom(q^dX,Y)$$ for the graded vector space of all morphisms.

For us a 2-category will always be a strict 2-category, which is the same thing as a category enriched in categories, i.e. the homomorphisms between two objects forms a category.

\section{The 2-category}\label{sec:2cat}

Let $k$ be a field.
 For each ordered pair $(s,t)$ of distinct elements of $S$, let $v_{st}$ be a nonzero element of $k$. To this data, there is a Kac-Moody 2-category $\U$, defined using a diagrammatic presentation. It has:

{\bf Objects:} $\la\in P$.

{\bf Generating 1-morphisms}:
\[
 \E_s1_\la:\la\to \la+\a_s,\qquad \F_s1_\la:\la\to \la-\a_s
\]

{\bf Generating 2-morphisms:}
\begin{equation*}\label{udotgens}
x\in \End(\E_s1_\la), \ \tau\in \Hom(\E_s\E_t1_\la,\E_t\E_s1_\la), \ \eta\in \Hom(\id_\la,\F_s\E_s1_\la),\mbox{ and } \epsilon\in \Hom(\F_s\E_s1_\la,\id_\la)
\end{equation*}
for all choices of $s,t\in S$ and $\la\in P$, 
drawn diagrammatically as
\begin{align}\label{solid1}
x 
&= 
\mathord{
\begin{tikzpicture}[baseline = 0]
	\draw[->,thick,darkred] (0.08,-.3) to (0.08,.4);
      \node at (0.08,0.05) {$\bullet$};
   \node at (0.08,-.4) {$\scriptstyle{s}$};
\end{tikzpicture}
}
{\scriptstyle\lambda}\:,
\qquad
\tau
= 
\mathord{
\begin{tikzpicture}[baseline = 0]
	\draw[->,thick,darkred] (0.28,-.3) to (-0.28,.4);
	\draw[->,thick,darkred] (-0.28,-.3) to (0.28,.4);
   \node at (-0.28,-.4) {$\scriptstyle{s}$};
   \node at (0.28,-.4) {$\scriptstyle{t}$};
   \node at (.4,.05) {$\scriptstyle{\lambda}$};
\end{tikzpicture}
}\:,
\qquad
\eta
= 
\mathord{
\begin{tikzpicture}[baseline = 0]
	\draw[<-,thick,darkred] (0.4,0.3) to[out=-90, in=0] (0.1,-0.1);
	\draw[-,thick,darkred] (0.1,-0.1) to[out = 180, in = -90] (-0.2,0.3);
    \node at (-0.2,.4) {$\scriptstyle{s}$};
  \node at (0.3,-0.15) {$\scriptstyle{\lambda}$};
\end{tikzpicture}
}\:,\qquad
\epsilon
= 
\mathord{
\begin{tikzpicture}[baseline = 0]
	\draw[<-,thick,darkred] (0.4,-0.1) to[out=90, in=0] (0.1,0.3);
	\draw[-,thick,darkred] (0.1,0.3) to[out = 180, in = 90] (-0.2,-0.1);
    \node at (-0.2,-.2) {$\scriptstyle{s}$};
  \node at (0.3,0.4) {$\scriptstyle{\lambda}$};
\end{tikzpicture}
}.
\end{align}
These are subject to a list of relations which we will reproduce below. By work of Brundan \cite{brundan}, different choices of relations that appear in the literature give equivalent 2-categories.

Before we discuss the relations, we first explain how to compose 2-morphisms. Given two 2-morphisms, their composition is defined to be the 2-morphism obtained by stacking the first 2-morphism on top of the second, if the endpoints of the 2-morphisms match.

The first relations we impose are the isotopy relations. These state that two diagrams which are isotopic are equal. This, together with the definition of composition, imply that the identity morphisms can be drawn as vertical lines without dots or crossings.

Second, we have the quiver Hecke relations:
\begin{align}\label{qha}
\mathord{
\begin{tikzpicture}[baseline = 0]
	\draw[<-,thick,darkred] (0.25,.6) to (-0.25,-.2);
	\draw[->,thick,darkred] (0.25,-.2) to (-0.25,.6);
  \node at (-0.25,-.27) {$\scriptstyle{s}$};
   \node at (0.25,-.27) {$\scriptstyle{t}$};
  \node at (.3,.25) {$\scriptstyle{\lambda}$};
      \node at (-0.13,-0.02) {$\bullet$};
\end{tikzpicture}
}
-
\mathord{
\begin{tikzpicture}[baseline = 0]
	\draw[<-,thick,darkred] (0.25,.6) to (-0.25,-.2);
	\draw[->,thick,darkred] (0.25,-.2) to (-0.25,.6);
  \node at (-0.25,-.27) {$\scriptstyle{s}$};
   \node at (0.25,-.27) {$\scriptstyle{t}$};
  \node at (.3,.25) {$\scriptstyle{\lambda}$};
      \node at (0.13,0.42) {$\bullet$};
\end{tikzpicture}
}
&=
\mathord{
\begin{tikzpicture}[baseline = 0]
 	\draw[<-,thick,darkred] (0.25,.6) to (-0.25,-.2);
	\draw[->,thick,darkred] (0.25,-.2) to (-0.25,.6);
  \node at (-0.25,-.27) {$\scriptstyle{s}$};
   \node at (0.25,-.27) {$\scriptstyle{t}$};
  \node at (.3,.25) {$\scriptstyle{\lambda}$};
      \node at (-0.13,0.42) {$\bullet$};
\end{tikzpicture}
}
-
\mathord{
\begin{tikzpicture}[baseline = 0]
 	\draw[<-,thick,darkred] (0.25,.6) to (-0.25,-.2);
	\draw[->,thick,darkred] (0.25,-.2) to (-0.25,.6);
  \node at (-0.25,-.27) {$\scriptstyle{s}$};
   \node at (0.25,-.27) {$\scriptstyle{t}$};
  \node at (.3,.25) {$\scriptstyle{\lambda}$};
      \node at (0.13,-0.02) {$\bullet$};
\end{tikzpicture}
}
=
\left\{
\begin{array}{ll}
\mathord{
\begin{tikzpicture}[baseline = 0]
 	\draw[->,thick,darkred] (0.08,-.3) to (0.08,.4);
	\draw[->,thick,darkred] (-0.28,-.3) to (-0.28,.4);
   \node at (-0.28,-.4) {$\scriptstyle{s}$};
   \node at (0.08,-.4) {$\scriptstyle{t}$};
 \node at (.28,.06) {$\scriptstyle{\lambda}$};
\end{tikzpicture}
}
&\text{if $s=t$,}\\
0&\text{otherwise,}\\
\end{array}
\right.
\end{align}
\begin{align}
\mathord{
\begin{tikzpicture}[baseline = 0]
	\draw[->,thick,darkred] (0.28,.4) to[out=90,in=-90] (-0.28,1.1);
	\draw[->,thick,darkred] (-0.28,.4) to[out=90,in=-90] (0.28,1.1);
	\draw[-,thick,darkred] (0.28,-.3) to[out=90,in=-90] (-0.28,.4);
	\draw[-,thick,darkred] (-0.28,-.3) to[out=90,in=-90] (0.28,.4);
  \node at (-0.28,-.4) {$\scriptstyle{s}$};
  \node at (0.28,-.4) {$\scriptstyle{t}$};
   \node at (.43,.4) {$\scriptstyle{\lambda}$};
\end{tikzpicture}
}
=
\left\{
\begin{array}{ll}
0&\text{if $i=j$,}\\
v_{st}\mathord{
\begin{tikzpicture}[baseline = 0]
	\draw[->,thick,darkred] (0.08,-.3) to (0.08,.4);
	\draw[->,thick,darkred] (-0.28,-.3) to (-0.28,.4);
   \node at (-0.28,-.4) {$\scriptstyle{s}$};
   \node at (0.08,-.4) {$\scriptstyle{t}$};
   \node at (.3,.05) {$\scriptstyle{\lambda}$};
\end{tikzpicture}
}&\text{if $a_{st}=0$,}\\
 v_{st}
\mathord{
\begin{tikzpicture}[baseline = 0]
	\draw[->,thick,darkred] (0.08,-.3) to (0.08,.4);
	\draw[->,thick,darkred] (-0.28,-.3) to (-0.28,.4);
   \node at (-0.28,-.4) {$\scriptstyle{s}$};
   \node at (0.08,-.4) {$\scriptstyle{t}$};
   \node at (.3,-.05) {$\scriptstyle{\lambda}$};
      \node at (-0.28,0.05) {$\bullet$};
\end{tikzpicture}
}
+
v_{ts}
\mathord{
\begin{tikzpicture}[baseline = 0]
	\draw[->,thick,darkred] (0.08,-.3) to (0.08,.4);
	\draw[->,thick,darkred] (-0.28,-.3) to (-0.28,.4);
   \node at (-0.28,-.4) {$\scriptstyle{s}$};
   \node at (0.08,-.4) {$\scriptstyle{t}$};
   \node at (.3,-.05) {$\scriptstyle{\lambda}$};
     \node at (0.08,0.05) {$\bullet$};
\end{tikzpicture}
}
&\text{otherwise,}\\
\end{array}
\right. \label{now}
\end{align}

\begin{align}
\mathord{
\begin{tikzpicture}[baseline = 0]
	\draw[<-,thick,darkred] (0.45,.8) to (-0.45,-.4);
	\draw[->,thick,darkred] (0.45,-.4) to (-0.45,.8);
        \draw[-,thick,darkred] (0,-.4) to[out=90,in=-90] (-.45,0.2);
        \draw[->,thick,darkred] (-0.45,0.2) to[out=90,in=-90] (0,0.8);
   \node at (-0.45,-.45) {$\scriptstyle{s}$};
   \node at (0,-.45) {$\scriptstyle{t}$};
  \node at (0.45,-.45) {$\scriptstyle{u}$};
   \node at (.5,-.1) {$\scriptstyle{\lambda}$};
\end{tikzpicture}
}
\!-
\!\!\!
\mathord{
\begin{tikzpicture}[baseline = 0]
	\draw[<-,thick,darkred] (0.45,.8) to (-0.45,-.4);
	\draw[->,thick,darkred] (0.45,-.4) to (-0.45,.8);
        \draw[-,thick,darkred] (0,-.4) to[out=90,in=-90] (.45,0.2);
        \draw[->,thick,darkred] (0.45,0.2) to[out=90,in=-90] (0,0.8);
   \node at (-0.45,-.45) {$\scriptstyle{s}$};
   \node at (0,-.45) {$\scriptstyle{t}$};
  \node at (0.45,-.45) {$\scriptstyle{u}$};
   \node at (.5,-.1) {$\scriptstyle{\lambda}$};
\end{tikzpicture}
}
&=
\left\{
\begin{array}{ll}
 v_{st}
\mathord{
\begin{tikzpicture}[baseline = 0]
	\draw[->,thick,darkred] (0.44,-.3) to (0.44,.4);
	\draw[->,thick,darkred] (0.08,-.3) to (0.08,.4);
	\draw[->,thick,darkred] (-0.28,-.3) to (-0.28,.4);
   \node at (-0.28,-.4) {$\scriptstyle{s}$};
   \node at (0.08,-.4) {$\scriptstyle{t}$};
   \node at (0.44,-.4) {$\scriptstyle{u}$};
  \node at (.6,-.1) {$\scriptstyle{\lambda}$};
     \node at (-0.28,0.05) {$\bullet$};
\end{tikzpicture}
}
+ v_{st}
\mathord{
\begin{tikzpicture}[baseline = 0]
	\draw[->,thick,darkred] (0.44,-.3) to (0.44,.4);
	\draw[->,thick,darkred] (0.08,-.3) to (0.08,.4);
	\draw[->,thick,darkred] (-0.28,-.3) to (-0.28,.4);
   \node at (-0.28,-.4) {$\scriptstyle{s}$};
   \node at (0.08,-.4) {$\scriptstyle{t}$};
   \node at (0.44,-.4) {$\scriptstyle{u}$};
  \node at (.6,-.1) {$\scriptstyle{\lambda}$};
     \node at (0.44,0.05) {$\bullet$};
\end{tikzpicture}
}
&\text{if $s=u \neq t$,}\\
0&\text{otherwise.}
\end{array}\label{qhalast}
\right.\end{align}

Third, we have the right adjunction relations
\begin{align}\label{rightadj}
\mathord{
\begin{tikzpicture}[baseline = 0]
  \draw[->,thick,darkred] (0.3,0) to (0.3,.4);
	\draw[-,thick,darkred] (0.3,0) to[out=-90, in=0] (0.1,-0.4);
	\draw[-,thick,darkred] (0.1,-0.4) to[out = 180, in = -90] (-0.1,0);
	\draw[-,thick,darkred] (-0.1,0) to[out=90, in=0] (-0.3,0.4);
	\draw[-,thick,darkred] (-0.3,0.4) to[out = 180, in =90] (-0.5,0);
  \draw[-,thick,darkred] (-0.5,0) to (-0.5,-.4);
   \node at (-0.5,-.5) {$\scriptstyle{s}$};
   \node at (0.5,0) {$\scriptstyle{\lambda}$};
\end{tikzpicture}
}
&=
\mathord{\begin{tikzpicture}[baseline=0]
  \draw[->,thick,darkred] (0,-0.4) to (0,.4);
   \node at (0,-.5) {$\scriptstyle{s}$};
   \node at (0.2,0) {$\scriptstyle{\lambda}$};
\end{tikzpicture}
},\qquad
\mathord{
\begin{tikzpicture}[baseline = 0]
  \draw[->,thick,darkred] (0.3,0) to (0.3,-.4);
	\draw[-,thick,darkred] (0.3,0) to[out=90, in=0] (0.1,0.4);
	\draw[-,thick,darkred] (0.1,0.4) to[out = 180, in = 90] (-0.1,0);
	\draw[-,thick,darkred] (-0.1,0) to[out=-90, in=0] (-0.3,-0.4);
	\draw[-,thick,darkred] (-0.3,-0.4) to[out = 180, in =-90] (-0.5,0);
  \draw[-,thick,darkred] (-0.5,0) to (-0.5,.4);
   \node at (-0.5,.5) {$\scriptstyle{s}$};
   \node at (0.5,0) {$\scriptstyle{\lambda}$};
\end{tikzpicture}
}
=
\mathord{\begin{tikzpicture}[baseline=0]
  \draw[<-,thick,darkred] (0,-0.4) to (0,.4);
   \node at (0,.5) {$\scriptstyle{s}$};
   \node at (0.2,0) {$\scriptstyle{\lambda}$};
\end{tikzpicture}
},
\end{align}

In order to describe the remaining relations, we first introduce a new 2-morphism.
\[
\mathord{
\begin{tikzpicture}[baseline = 0]
	\draw[<-,thick,darkred] (0.28,-.3) to (-0.28,.4);
	\draw[->,thick,darkred] (-0.28,-.3) to (0.28,.4);
   \node at (-0.28,-.4) {$\scriptstyle{t}$};
   \node at (-0.28,.5) {$\scriptstyle{s}$};
   \node at (.4,.05) {$\scriptstyle{\lambda}$};
\end{tikzpicture}
}
:=
\mathord{
\begin{tikzpicture}[baseline = 0]
	\draw[->,thick,darkred] (0.3,-.5) to (-0.3,.5);
	\draw[-,thick,darkred] (-0.2,-.2) to (0.2,.3);
        \draw[-,thick,darkred] (0.2,.3) to[out=50,in=180] (0.5,.5);
        \draw[->,thick,darkred] (0.5,.5) to[out=0,in=90] (0.8,-.5);
        \draw[-,thick,darkred] (-0.2,-.2) to[out=230,in=0] (-0.5,-.5);
        \draw[-,thick,darkred] (-0.5,-.5) to[out=180,in=-90] (-0.8,.5);
  \node at (-0.8,.6) {$\scriptstyle{s}$};
   \node at (0.28,-.6) {$\scriptstyle{t}$};
   \node at (1.05,.05) {$\scriptstyle{\lambda}$};
\end{tikzpicture}
}
:\E_t \F_s 1_\lambda \to \F_s \E_t 1_\lambda.
\]
To finish, we require that the following $2$-morphisms are isomorphisms:
\begin{align}\label{rightcross}
\mathord{
\begin{tikzpicture}[baseline = 0]
	\draw[<-,thick,darkred] (0.28,-.3) to (-0.28,.4);
	\draw[->,thick,darkred] (-0.28,-.3) to (0.28,.4);
   \node at (-0.28,-.4) {$\scriptstyle{t}$};
   \node at (-0.28,.5) {$\scriptstyle{s}$};
   \node at (.4,.05) {$\scriptstyle{\lambda}$};
\end{tikzpicture}
}
&:\E_t \F_s 1_\lambda \stackrel{\sim}{\rightarrow} \F_s \E_t 1_\lambda
&\text{if $s \neq t$,}\\
\label{inv2}
\mathord{
\begin{tikzpicture}[baseline = 0]
	\draw[<-,thick,darkred] (0.28,-.3) to (-0.28,.4);
	\draw[->,thick,darkred] (-0.28,-.3) to (0.28,.4);
   \node at (-0.28,-.4) {$\scriptstyle{s}$};
   \node at (-0.28,.5) {$\scriptstyle{s}$};
   \node at (.4,.05) {$\scriptstyle{\lambda}$};
\end{tikzpicture}
}
\oplus
\bigoplus_{n=0}^{\langle \a_s^\vee,\lambda\rangle-1}
\mathord{
\begin{tikzpicture}[baseline = 0]
	\draw[<-,thick,darkred] (0.4,0) to[out=90, in=0] (0.1,0.4);
	\draw[-,thick,darkred] (0.1,0.4) to[out = 180, in = 90] (-0.2,0);
    \node at (-0.2,-.1) {$\scriptstyle{s}$};
  \node at (0.3,0.5) {$\scriptstyle{\lambda}$};
      \node at (-0.3,0.2) {$\scriptstyle{n}$};
      \node at (-0.15,0.2) {$\bullet$};
\end{tikzpicture}
}
&:
\E_s \F_s 1_\lambda\stackrel{\sim}{\rightarrow}
\F_s \E_s 1_\lambda \oplus 1_\lambda^{\oplus \langle \a_s^\vee,\lambda\rangle}
&\text{if $\langle \a_s^\vee,\lambda\rangle \geq
  0$},\\
\mathord{
\begin{tikzpicture}[baseline = 0]
	\draw[<-,thick,darkred] (0.28,-.3) to (-0.28,.4);
	\draw[->,thick,darkred] (-0.28,-.3) to (0.28,.4);
   \node at (-0.28,-.4) {$\scriptstyle{s}$};
   \node at (-0.28,.5) {$\scriptstyle{s}$};
   \node at (.4,.05) {$\scriptstyle{\lambda}$};
\end{tikzpicture}
}
\oplus
\bigoplus_{n=0}^{-\langle \a_s^\vee,\lambda\rangle-1}
\mathord{
\begin{tikzpicture}[baseline = 0]
	\draw[<-,thick,darkred] (0.4,0.2) to[out=-90, in=0] (0.1,-.2);
	\draw[-,thick,darkred] (0.1,-.2) to[out = 180, in = -90] (-0.2,0.2);
    \node at (-0.2,.3) {$\scriptstyle{s}$};
  \node at (0.3,-0.25) {$\scriptstyle{\lambda}$};
      \node at (0.55,0) {$\scriptstyle{n}$};
      \node at (0.38,0) {$\bullet$};
\end{tikzpicture}
}
&
:\E_s \F_s 1_\lambda \oplus 
1_\lambda^{\oplus -\langle \a_s^\vee,\lambda\rangle}
\stackrel{\sim}{\rightarrow}
 \F_s \E_s 1_\lambda&\text{if $\langle \a_s^\vee,\lambda\rangle \leq
  0$}.\label{inv3}
\end{align}
where the label $n$ next to a dot means that this strand carries $n$ dots.

In more detail, the meaning of these last relations is
that $\U$ contains additional generators which are
two-sided inverses to the $2$-morphisms described in (\ref{rightcross})--(\ref{inv3}). 
They are $L\in \Hom(\F_s\E_t 1_\la, \E_t\F_s 1_\la)$, $\psi_n\in\Hom(\id_\la,\E_s\F_s 1_{\la})$ for $0\leq n<\langle\a_s^\vee,\la\rangle$ and $\varphi_m\in \Hom(\F_s\E_s1_\la,\id_\la)$ for $0\leq m<-\langle \a_s^\vee,\la\rangle$, and
we draw them as
\begin{align*}
L&=
\mathord{
\begin{tikzpicture}[baseline = 0]
	\draw[->,thick,darkred] (0.18,-.15) to (-0.18,.3);
	\draw[<-,thick,darkred] (-0.18,-.15) to (0.18,.3);
   \node at (0.18,.4) {$\scriptstyle{s}$};
   \node at (0.18,-.25) {$\scriptstyle{t}$};
\end{tikzpicture}
}\!
{\scriptstyle\la}\: \qquad
\psi_n=
\mathord{
\begin{tikzpicture}[baseline = -2]
	\draw[-,thick,darkred] (0.3,0.2) to[out=-90, in=0] (0.1,-0.1);
	\draw[->,thick,darkred] (0.1,-0.1) to[out = 180, in = -90] (-0.1,0.2);
    \node at (0.3,.3) {$\scriptstyle{s}$};
      \node at (0.1,-0.27) {$\scriptstyle{n}$};
      \node at (0.1,-0.09) {$\scriptscriptstyle{\spadesuit}$};
\end{tikzpicture}
}\!
{\scriptstyle\lambda}\: \qquad
\varphi_m=
\mathord{
\begin{tikzpicture}[baseline = -3]
	\draw[-,thick,darkred] (0.3,-0.1) to[out=90, in=0] (0.1,0.2);
	\draw[->,thick,darkred] (0.1,0.2) to[out = 180, in = 90] (-0.1,-0.1);
    \node at (0.3,-.2) {$\scriptstyle{s}$};
      \node at (0.12,.35999) {$\scriptstyle{m}$};
      \node at (0.1,0.2) {$\scriptscriptstyle{\spadesuit}$};
\end{tikzpicture}
}
\,{\scriptstyle\la}\: .
\end{align*}
%
They satisfy the following relations: 
\begin{align*}
\mathord{
\begin{tikzpicture}[baseline = 0]
	\draw[->,thick,darkred] (0.28,-.3) to (-0.28,.4);
	\draw[<-,thick,darkred] (-0.28,-.3) to (0.28,.4);
   \node at (0.28,-.4) {$\scriptstyle{t}$};
   \node at (0.28,.5) {$\scriptstyle{s}$};
   \node at (.4,.05) {$\scriptstyle{\lambda}$};
\end{tikzpicture}
}
&= \bigg(\mathord{
\begin{tikzpicture}[baseline = 0]
	\draw[<-,thick,darkred] (0.28,-.3) to (-0.28,.4);
	\draw[->,thick,darkred] (-0.28,-.3) to (0.28,.4);
   \node at (-0.28,-.4) {$\scriptstyle{t}$};
   \node at (-0.28,.5) {$\scriptstyle{s}$};
   \node at (.4,.05) {$\scriptstyle{\lambda}$};
\end{tikzpicture}
}
\bigg)^{-1}
&&\text{if $s \neq t$,}\\
-\mathord{
\begin{tikzpicture}[baseline = 0]
	\draw[->,thick,darkred] (0.28,-.3) to (-0.28,.4);
	\draw[<-,thick,darkred] (-0.28,-.3) to (0.28,.4);
   \node at (0.28,-.4) {$\scriptstyle{s}$};
   \node at (0.28,.5) {$\scriptstyle{s}$};
   \node at (.4,.05) {$\scriptstyle{\lambda}$};
\end{tikzpicture}
}\oplus\bigoplus_{n=0}^{\langle \a_s^\vee,\lambda\rangle-1}
\mathord{
\begin{tikzpicture}[baseline = 0]
	\draw[-,thick,darkred] (0.4,0.4) to[out=-90, in=0] (0.1,0);
	\draw[->,thick,darkred] (0.1,0) to[out = 180, in = -90] (-0.2,0.4);
    \node at (0.4,.5) {$\scriptstyle{s}$};
  \node at (0.5,0.15) {$\scriptstyle{\lambda}$};
      \node at (0.12,-0.2) {$\scriptstyle{n}$};
      \node at (0.12,0.01) {$\scriptstyle{\spadesuit}$};
\end{tikzpicture}
}
&=
\bigg(
\!\mathord{
\begin{tikzpicture}[baseline = 0]
	\draw[<-,thick,darkred] (0.28,-.3) to (-0.28,.4);
	\draw[->,thick,darkred] (-0.28,-.3) to (0.28,.4);
   \node at (-0.28,-.4) {$\scriptstyle{s}$};
   \node at (-0.28,.5) {$\scriptstyle{s}$};
   \node at (.4,.05) {$\scriptstyle{\lambda}$};
\end{tikzpicture}
}
\oplus
\bigoplus_{n=0}^{\langle \a_s^\vee,\lambda\rangle-1}
\mathord{
\begin{tikzpicture}[baseline = 0]
	\draw[<-,thick,darkred] (0.4,0) to[out=90, in=0] (0.1,0.4);
	\draw[-,thick,darkred] (0.1,0.4) to[out = 180, in = 90] (-0.2,0);
    \node at (-0.2,-.1) {$\scriptstyle{s}$};
  \node at (0.3,0.5) {$\scriptstyle{\lambda}$};
      \node at (-0.3,0.2) {$\scriptstyle{n}$};
      \node at (-0.15,0.2) {$\bullet$};
\end{tikzpicture}
}
\bigg)^{-1}
&&
\text{if $\langle \a_s^\vee,\lambda \rangle \geq 0$,}\\
-\mathord{
\begin{tikzpicture}[baseline = 0]
	\draw[->,thick,darkred] (0.28,-.3) to (-0.28,.4);
	\draw[<-,thick,darkred] (-0.28,-.3) to (0.28,.4);
   \node at (0.28,-.4) {$\scriptstyle{s}$};
   \node at (0.28,.5) {$\scriptstyle{s}$};
   \node at (.4,.05) {$\scriptstyle{\lambda}$};
\end{tikzpicture}
}\oplus\bigoplus_{m=0}^{-\langle \a_s^\vee,\lambda\rangle-1}
\mathord{
\begin{tikzpicture}[baseline = 0]
	\draw[-,thick,darkred] (0.4,0) to[out=90, in=0] (0.1,0.4);
	\draw[->,thick,darkred] (0.1,0.4) to[out = 180, in = 90] (-0.2,0);
    \node at (0.4,-.1) {$\scriptstyle{s}$};
  \node at (0.5,0.45) {$\scriptstyle{\lambda}$};
      \node at (0.12,0.6) {$\scriptstyle{m}$};
      \node at (0.12,0.4) {$\scriptstyle{\clubsuit}$};
\end{tikzpicture}
}
&=
\bigg(
\!\mathord{
\begin{tikzpicture}[baseline = 0]
	\draw[<-,thick,darkred] (0.28,-.3) to (-0.28,.4);
	\draw[->,thick,darkred] (-0.28,-.3) to (0.28,.4);
   \node at (-0.28,-.4) {$\scriptstyle{s}$};
   \node at (-0.28,.5) {$\scriptstyle{s}$};
   \node at (.4,.05) {$\scriptstyle{\lambda}$};
\end{tikzpicture}
}
\oplus
\!\bigoplus_{m=0}^{-\langle h_i,\lambda\rangle-1}
\mathord{
\begin{tikzpicture}[baseline = 0]
	\draw[<-,thick,darkred] (0.4,0.4) to[out=-90, in=0] (0.1,0);
	\draw[-,thick,darkred] (0.1,0) to[out = 180, in = -90] (-0.2,0.4);
    \node at (-0.2,.5) {$\scriptstyle{s}$};
  \node at (0.4,-0.05) {$\scriptstyle{\lambda}$};
      \node at (0.56,0.2) {$\scriptstyle{m}$};
      \node at (0.37,0.2) {$\bullet$};
\end{tikzpicture}
}
\bigg)^{-1}
&&
\text{if $\langle \a_s^\vee ,\lambda \rangle \leq 0$.}
\end{align*}

This concludes the definition of $\U$. The morphism spaces in $\U$ are all $\Z$-graded, the generators in (\ref{solid1}) have degrees $2$, $a_{st}$, $1+\langle \a_s^\vee \la\rangle$ and $1-\langle \a_s^\vee \la\rangle$ respectively.

Our next goal is to introduce 
 some more 2-morphisms in order to state a non-degeneracy theorem for $\U$.
 
The leftward caps are defined by
\[
\mathord{
\begin{tikzpicture}[baseline = 0]
	\draw[-,thick,darkred] (0.4,0.4) to[out=-90, in=0] (0.1,0);
	\draw[->,thick,darkred] (0.1,0) to[out = 180, in = -90] (-0.2,0.4);
    \node at (0.4,.55) {$\scriptstyle{s}$};
  \node at (0.5,0.15) {$\scriptstyle{\lambda}$};
\end{tikzpicture}
}
= \left\{
 \begin{array}{lll}
 \mathord{
\begin{tikzpicture}[baseline = 0]
	\draw[-,thick,darkred] (0.4,0.4) to[out=-90, in=0] (0.1,0);
	\draw[->,thick,darkred] (0.1,0) to[out = 180, in = -90] (-0.2,0.4);
    \node at (0.4,.5) {$\scriptstyle{s}$};
  \node at (0.5,0.15) {$\scriptstyle{\lambda}$};
      \node at (0.12,-0.3) {$\scriptstyle{\langle \a_s^\vee,\la \rangle -1}$};
      \node at (0.12,0.01) {$\scriptstyle{\spadesuit}$};
\end{tikzpicture}
}
 &&
\text{if $\langle \a_s^\vee,\lambda \rangle > 0$.} \\
\mathord{
\begin{tikzpicture}[baseline = 0]
	\draw[-,thick,darkred] (0.28,.6) to[out=240,in=90] (-0.28,-.1);
	\draw[<-,thick,darkred] (-0.28,.6) to[out=300,in=90] (0.28,-0.1);
   \node at (0.28,.7) {$\scriptstyle{s}$};
   \node at (.45,.2) {$\scriptstyle{\lambda}$};
	\draw[-,thick,darkred] (0.28,-0.1) to[out=-90, in=0] (0,-0.4);
	\draw[-,thick,darkred] (0,-0.4) to[out = 180, in = -90] (-0.28,-0.1);
      \node at (0.89,-0.1) {$\scriptstyle{-\langle \a_s^\vee,\lambda\rangle}$};
      \node at (0.27,-0.1) {$\bullet$};
\end{tikzpicture}
}
&&
\text{if $\langle \a_s^\vee,\lambda \rangle \leq  0$.}
\end{array}
\right.
 \]


From these, we can define clockwise bubbles with a non-negative number of dots $\mathord{
\begin{tikzpicture}[baseline = 0]
  \draw[<-,thick,darkred] (0,0.4) to[out=180,in=90] (-.2,0.2);
  \draw[-,thick,darkred] (0.2,0.2) to[out=90,in=0] (0,.4);
 \draw[-,thick,darkred] (-.2,0.2) to[out=-90,in=180] (0,0);
  \draw[-,thick,darkred] (0,0) to[out=0,in=-90] (0.2,0.2);
 \node at (0,-.1) {$\scriptstyle{s}$};
   \node at (0.3,0.2) {$\scriptstyle{\lambda}$};
   \node at (-0.2,0.2) {$\bullet$};
   \node at (-0.4,0.2) {$\scriptstyle{r}$};
\end{tikzpicture}
}$, which are endomorphisms of $\operatorname{id}_\la$. It is convenient to extend this definition to allow a negative number of dots, as in \cite[Eq 3.8]{brundan}, a convention which we will employ.

Let $\B_\la$ be the unital commutative $k$-algebra freely generated by all bubbles  $\mathord{
\begin{tikzpicture}[baseline = 0]
  \draw[<-,thick,darkred] (0,0.4) to[out=180,in=90] (-.2,0.2);
  \draw[-,thick,darkred] (0.2,0.2) to[out=90,in=0] (0,.4);
 \draw[-,thick,darkred] (-.2,0.2) to[out=-90,in=180] (0,0);
  \draw[-,thick,darkred] (0,0) to[out=0,in=-90] (0.2,0.2);
 \node at (0,-.1) {$\scriptstyle{s}$};
   \node at (0.3,0.2) {$\scriptstyle{\lambda}$};
   \node at (-0.2,0.2) {$\bullet$};
   \node at (-0.4,0.2) {$\scriptstyle{r}$};
\end{tikzpicture}
}$, for $r\geq \langle \a_s^\vee,\la\rangle $ ($r$ here may be negative). Note that all of these generators are in positive degree. So $\B_\la$ is non-negatively graded and in degree zero is simply the ground field $k$.

We are now in a position to define the notion of nondegeneracy. Let $\ii=(i_1,i_2,\ldots,i_n)$ and $\jj=(j_1,j_2,\ldots,j_n)$ be two sequences of elements of $S$. Let $D$ be a matching of $\ii$ and $\jj$. To this matching, choose a representative diagram comprised of downward pointing crossings where any pair of strings crosses at most twice - this represents an element of $\HOM (\F_{i_1}\F_{i_2}\cdots\F_{i_n},\F_{j_1}\F_{j_2}\cdots \F_{j_n})$. Now given a matching $D$ and a sequence $\aa=(a_1,a_2,\ldots,a_n)$ of natural numbers, create a morphism $f_{D,\aa}$ by adding $a_i$ dots onto the bottom of the $i$-th strand from the left of this representative diagram for $1\leq i \leq n$. Let $\Pi$ be a basis for $\B_\la$.

\begin{definition}
The 2-category $\U$ is said to be nondegenerate if for all $\ii,\jj$, the set 
\[
\{f_{D,\aa}\otimes \pi \mid D\in D(\ii,\jj),\aa\in \N^{|\ii|}, \pi\in \Pi\}
\]
is a basis for $\HOM (\F_{i_1}\F_{i_2}\cdots\F_{i_n},\F_{j_1}\F_{j_2}\cdots \F_{j_n})$.
\end{definition}

The following fundamental theorem is important to us:

\begin{theorem}[Nondegeneracy theorem]
\label{nondegeneracy}
The 2-category $\U$ is nondegenerate.
\end{theorem}

This result was proved by Khovanov and Lauda \cite[Theorem 1.3]{kl3} in type $A$, and by Webster \cite[Theorem A]{unfurling} and DuPont \cite[Theorem 3.5.3]{dupont} in general. Although only stated as a result about homomorphism spaces between products of $\F$'s, standard adjunction techniques allow us to find from this a basis of all morphism spaces in $\U$.

Out of $\U$, for $*\in \{b,+,-,\emptyset\}$ we construct a new 2-category $K^*(\udot)$. Like $\U$, its objects are elements of $P$. For any two elements $\la,\mu\in P$, the morphism category in $K^*(\udot)$ is defined to be
\[
 \Hom_{K^*(\udot)}(\la,\mu)=\operatorname{Kar}(K^*(\Hom_{\U}(\la,\mu)))
\]
 where $K^\ast$ stands for taking the (bounded, bounded above, bounded below or unbounded) homotopy category and $\operatorname{Kar}$ refers to taking the Karoubian envelope.

\section{Rouquier Complexes}

Let $n$ be an integer. The Nil-Hecke algebra is the graded algebra generated by elements $x_1,x_2,\ldots,x_n$ in degree 2, and $\partial_1,\partial_2,\ldots,\partial_{n-1}$ in degree $-2$, subject to the following relations:
\begin{align*}
x_i x_j&=x_j x_i \\
\partial_i^2=0 \\
\partial_i \partial_{i+1}\partial_i &= \partial_{i+1}\partial_i\partial_{i+1} \\
\partial_i x_j &= x_j \partial_i  \quad \text{if}\quad |i-j|>1 \\
\partial_i x_{i+1} &= x_{i+1}\partial_i \pm 1 \\
\partial_{i+1} x_i &= x_i\partial_{i+1} \pm 1.
\end{align*}
It is known that there is an isomorphism of algebras
\[
NH_n \cong \operatorname{Mat}_{[n]!}(k[x_1,\ldots,x_n]^{S_n}).
\]
Let $e_n=x_2x_3^2,\ldots x_n^{n-1}(\partial_1\partial_2\cdots\partial_{n-1})(\partial_1\partial_2\cdots\partial_{n-2})\cdots(\partial_1\partial_2)(\partial_1)$. Then $e_n$ is a primitive idempotent in $NH_n$. Let $s\in S$. There is a natural map $\psi:NH_n\to \End(\F_s^n)$ sending $x_i$ to the dot acting on the $i$-th factor and $\partial_i$ to the crossing between the $i$-th and $(i+1)$-st factor. Inside $\udot$, we define the divided power
\[
\F_s^{(n)} = q^{-\frac{n(n-1)}{2}}\psi(e_n)\F_s^n
\]
and similarly for $\E_s^{(n)}$. If $n<0$, we say these divided powers are zero.

For each $s\in S$ and $\la\in P$, we define the Rickard complex $\Theta_s 1_\la$ to be the complex of objects of $1_{s\la}\U 1_\la$ whose $r$-th component is
\[
\Theta^r_s 1_\la = q^{-r} F_s^{(\langle \a_s^\vee,\la \rangle +r)}E_s^{(r)}.
\]
The differentials of $\Theta_s 1_\la$ are given by the composition of the counit $\id \to \F_s\E_s$ together with projection by appropriate idempotents. The most important result about these Rickard complexes is the following:

\begin{theorem}\cite[Theorem 5.13]{vera}
The complex $\Theta_s$ is invertible in $K(\udot)$.
\end{theorem}

From the above result, we define a triangulated autoequivalence $\T_s$ of $K^b(\udot)$ for each $s\in S$ by conjugating by the Rickard complex $\Th_s$.

\begin{theorem}\cite{allr}\label{aller}
If $s$ and $t$ are connected by an edge, then
\begin{equation}\label{tsdef}
 \T_s(\F_t\one_\la)= q\F_t \F_s 1_{s\la} \to \clubsuit \F_s \F_t 1_{s\la},
\end{equation}
\begin{equation}\label{tsinv}
 \T_s(q\F_s \F_t 1_{\la} \to \clubsuit \F_t \F_s 1_\la) = \F_t 1_{s\la}
\end{equation}
where the differential in each case is the downward crossing, and the
symbol $\clubsuit$ is used to denote the term in homological degree zero.
\end{theorem}

We note that this is proved directly in \cite{allr}, but computations to prove this also already appear in the literature in \cite[Corollary 5.4]{ck}.

\begin{remark}
At the level of the Grothendieck group, $\T_s$ decategorifies the braid group automorphism $T_s$ constructed in \cite[\S 38]{bookoflusztig}.
\end{remark}


\section{Quiver Hecke Algebras}

The quiver Hecke algebra $R$ can be defined inside $\U$ as the algebra generated by the upward strands, subject to the isotopy and quiver Hecke relations. There are no identities in $R$ that do not follow from the relations only involving upward strands.

For $\nu\in\N I$, let $R(\nu)$ be the subalgebra of $R$ consisting of diagrams where the colours of the strands add to $\nu$. Here we refer to a strand labelled by $s\in S$ as having colour $s$.

The algebra $R(\nu)$ is {\emph Laurentian}, this means that for all $n$, $R(\nu)_n$ is finite dimensional and for $n\ll 0$, we have $R(\nu)_n=0$.

By placing strands next to each other there is a nonunital inclusion of algebras $R(\la)\otimes R(\mu)\hookrightarrow R(\la+\mu)$.
There is thus a corresponding induction functor
\[
\Ind\map{R(\la)\mods \times R(\mu)\mods}{R(\la+\mu)\mods}
 \]
written $(M,N)\mapsto M\circ N$, given by
\[
 M\circ N:=R(\la+\mu)e\otimes_{R(\la)\otimes R(\mu)}(M\otimes N)
\]
where $e$ is the image of the unit under the algebra inclusion.

Its left adjoint is the restriction functor
\[
\Res_{\la\mu}\map{R(\la+\mu)\mods}{R(\la)\otimes R(\mu)\mods}
\]
given by
\[
\Res_{\la\mu}(M)=eM.
\]

Let $\la,\mu\in \N I$. Then there is another adjunction \cite[Corollary 2.6]{klr1}:
\begin{equation}\label{otheradjunction}
 \Ext^i(A,B\circ C)\cong q^{\lambda\cdot\mu}\Ext^i(\Res_{\mu\la} A, C\boxtimes B),
\end{equation}
where $B$ is an $R(\la)$-module, $C$ is an $R(\mu)$-module and $A$ is an $R(\la)$-module.

The main theorem of \cite{khovanovlauda} is that the quiver Hecke algebras categorify $\f$ (introduced in \S \ref{prelim}), which means there is an isomorphism
\begin{equation} \label{klmain}
\bigoplus_\nu K_0(R(\nu)\prmod)\cong \f.
\end{equation}
Here $R(\nu)\prmod$ is the category of finitely generated graded projective $R(\nu)$-modules. Further properties of this isomorphism including compatibilities with induction and restriction functors are discussed and proven in \cite{khovanovlauda}.

Let $\f^\ast$ be the graded $\Z[q,q\inv]$-dual of $\f$. As a consequence of (\ref{klmain}), it has the following categorical interpretation:
\[
\bigoplus_\nu K_0(R(\nu)\fmod)\cong \f^\ast.
\]
Here $R(\nu)\fmod$ denotes the category of finite dimensional graded $R(\nu)$-modules. Under these isomorphisms, the canonical pairing between $\f$ and $\f^\ast$ is incarnated by the Hom-pairing. We let $\th_s^\ast$ be the basis element of $\f^\ast$ dual to $\th_s$.

\section{The embeddings of categories}

For each $\la\in P$, there is a functor
\[
i_\la: K^-(R(\nu)\prmod) \rightarrow \Hom_{K^-(\udot)}(\la,\la-\nu).
\]

If $\ii=(i_1,i_2,\ldots,i_n)$ is a sequence of elements of $S$, then there is an idempotent $e_\ii\in R(i_1+i_2+\cdots+i_n)$ consisting of vertical strings with colours $i_1,\ldots,i_n$ in that order without any dots. We therefore get a projective $R(\nu)$-module $P_\ii=R(\nu)e_\ii=P_{i_1}\circ P_{i_2}\circ \cdots \circ P_{i_n}$. Under this functor, we have 
\[
i_\la(P_\ii)=\F_{i_1}\F_{i_2}\cdots \F_{i_n}{1}_\la.
\]


These functors satisfy a compatibility between induction and composition, namely that the following diagram commutes:
     \[
\begin{tikzcd}[swap]
    K^-(R(\mu)\prmod)\times K^-(R(\nu)\prmod) \arrow{r}[swap]{\Ind}{}
           \arrow{d}[swap]{i_{\la-\nu}\times i_\la}
  & K^-(R(\mu+\nu)\prmod) \arrow{d}[swap]{i_\la} \\  
    \Hom_{K^-(\udot)}(\la-\nu,\la-\mu-\nu)\times \Hom_{K^-(\udot)}(\la,\la-\nu) \arrow{r} 
  & \Hom_{K^-(\udot)}(\la,\la-\mu-\nu)
\end{tikzcd}
\]
where the horizontal map along the bottom row is the composition in $K^-(\udot)$.


Recall that for each $\la\in P$, we have a commutative algebra $\B_\la$ 
freely generated by clockwise bubbles with loops, introduced in \S \ref{sec:2cat}.
By the nondegeneracy theorem, $\B_\la\cong\End_\U(\id_\la)$.  
In degree zero, this algebra is spanned by the identity and in negative degrees, it is zero.

\begin{theorem}\label{embedding}
 Suppose $X$ and $Y$ are two objects in $K^b(R(\nu))\prmod)$. Then there is a canonical isomorphism of graded vector spaces.
 \[
   \Hom_{K^b(R(\nu)-\operatorname{pmod})}(X,Y)\otimes \B_\la\cong  \Hom_{K^-(\udot)}( i_\la(X),i_\la(Y))
 \]
\end{theorem}

\begin{proof}
Use $X=(P_\bullet,d)$ and $Y=(Q_\bullet,d)$ to denote bounded chain complexes of projective modules representing $X$ and $Y$.
Let $\{b_j\}_{j\in J}$ be a basis of $\B_\la$. 

We first prove surjectivity. Let $f_\bullet\map{i_\la(P_\bullet)}{i_\la(Q_\bullet)}$ be a morphism of chain complexes. We can express each $f_i:i_\la(P_i)\to i_\la(Q_i)$ in the form $f_i=\sum_{j\in J} g_{ij}\otimes h_j$ where each $g_{ij}$ has no bubbles, i.e. comes from a morphism from $P_i$ to $Q_i$.

The condition that $f_\bullet$ is a chain map is expressed in the identity
\[
 \sum_{j\in J} d g_{ij}\otimes b_j = \sum_{j\in J} g_{i+1,j} d \otimes b_j.
\]

By the nondegeneracy theorem, Theorem \ref{nondegeneracy}, this implies that for each $j$, $d g_{ij}=g_{i+1,j} d$.
Thus for each $j\in J$, the collection $g_{\bullet,j}$ is a chain map in $K^b(R(\nu)\prmod)$.

Since there are only a finite number of $j$ for which $g_{\bullet,j}$ is non-zero, we conclude that the map from $\Hom_{K^b(R(\nu)-\operatorname{pmod}}(X,Y)\otimes\B_\la$ to $\Hom_{K^-(\udot)}(i_\la(X),i_\la(Y))$ is surjective.

We now show injectivity. Suppose we have an element in   $\Hom(X,Y)_{K^b(R(\nu)-\operatorname{pmod}}\otimes \B_\la$ that maps to zero. Write it in the form $\sum_{j\in J}f_{\bullet j}\otimes b_j$ with $f_{\bullet j}:X\to Y$ a chain map for all $j$. Since it gets sent to zero in $\Hom_{K^-(\udot)}( i_\la(X),i_\la(Y))$, there exists a chain homotopy $\sum_{j\in J}i_\la(h_{ij})\otimes b_j:i_\la(P_{i-1})\to i_\la(Q_i)$ such that 
\[
\sum_{j\in J} di_\la(h_{i+1,j})\otimes b_j + \sum_{j\in J} i_\la(h_{ij})d \otimes b_j = \sum_{j\in J} i_\la(f_{ij})\otimes b_j.
\]
By the non-degeneracy theorem, this implies that for each $j$, we have 
\[
dh_{i+1,j}+h_{ij}d = f_{ij}
\]
and hence $f_{\bullet j}$ is homotopic to the zero map. Therefore $\sum_{j\in J} f_{\bullet j}\otimes b_j$ is zero in the homotopy category, as required.
\end{proof}

\begin{corollary}
 The functor $i_\la$ is faithful.
\end{corollary}

\section{Standard modules}

We summarise the current state of the theory of standard modules for quiver Hecke algebras. This theory is currently known to exist in finite type in all characteristics and in symmetric affine type when $k$ is of characteristic zero. The references are \cite{bkm} in the former case and \cite{mcn3} in the latter.

Let $\Phi^+$ be the set of positive roots. A \emph{convex order} on $\Phi^+$ is a preorder $\prec$ such that
\begin{itemize}
 \item If $S$ and $T$ are two subsets of $\Phi^+$ such that $s\prec t$ for all $s\in S$ and $t\in T$ then
 \[
  \spann_{\R_{\geq 0}} S \cap \spann_{\R_{\geq 0}} T = \{0 \}, 
 \]

 \item If $s\preceq t$ and $t\preceq s$ then $s$ and $t$ are proportional.

\end{itemize}

Let $\a\in\Phi^+$. A representation $M$ of $R(\a)$ is said to be {\it semicuspidal} (with respect to the convex order $\prec$) if $\Res_{\b\ga}M\neq 0$ implies that $\b$ is a sum of roots less than $\a$ and $\ga$ is a sum of roots greater than $\a$.

Let $\a$ be an indivisible root. An indecomposable projective object in the category of semicuspidal $R(\a)$-modules is called a root module. The grading shift on these root modules is customarily normalised such that their heads are self-dual. For each indecomposable root $\a$, the number of root modules for $\a$ is equal to the dimension of the root space $\g_\a$. In particular, if $\a$ is a real root, there is a unique root module, which we call $\D(\a)$.

We consider the standard modules introduced in \cite{bkm} and \cite{mcn3}. These depend on the convex order $\prec$ and are built out of root modules. The root modules corresponding to real roots have already been introduced, these are the modules $\D(\a)$. For the indivisible imaginary root $\d$, we will call the modules denoted $\D(\w)$ in \cite{mcn3} root modules. These are the projective modules in the category of cuspidal $R(\d)$-modules. 

Standard modules are naturally indexed by root partitions. A root partition is a sequence $\la=(\a_1^{n_1},\cdots,\a_l^{n_l})$ where $\a_1\succ \cdots \succ \a_l$ are indivisible roots, each $n_i$ is a positive integer unless $\a_i=\delta$, in which case it is a collection of partitions, indexed naturally by a collection of chamber coweights adapted to $\prec $ (as defined in \cite[\S 12]{mcn3}). 
We also use the symbol $\prec$ to denote the lexicographical order on root partitions.

To each term $\a_i^{n_i}$ in a root partition, a standard module $\D(\a_i)^{(n_i)}$ is constructed. 
If $\a_i$ is real then $\D(\a_i)^{\circ n_i}$ is a direct sum of $n_i!$ copies of the module $\D(\a_i)^{(n_i)}$ with grading shifts.
If $\a_i$ is imaginary then $\D(\a_i)^{(n_i)}$ is a summand of an induction product of the imaginary root modules $\D(\w)$; see \cite[\S 19]{mcn3} for the details (where this module is denoted $\D(\underline{\la})$). The standard module is then defined to be the indecomposable module
\[
 \D(\la)=\D(\a_1)^{(n_1)}\circ\cdots\circ \D(\a_l)^{(n_l)}.
\]

In \cite{bkm} and \cite{mcn3} homological properties of these modules are developed which justify the use of the name standard. In this paper we call a module standard if it is isomorphic to any grading shift of any $\D(\la)$.

From their construction, the family of standard modules satisfies the following property:

\begin{proposition}\label{roottostandard}
 Every standard module is obtained from a root module by a process of induction and taking direct summands.
\end{proposition}

\begin{remark}
 We expect that if a theory of standard modules is developed in affine type over a field of positive characteristic, it will not satisfy Proposition \ref{roottostandard}. This expectation is because of the essential use of the semisimplicity of the representations of the symmetric group in \cite{mcn3}. Geometric evidence against a good theory of standard modules in positive characteristic is discussed in \cite{mm}.
 However, extensions to non-symmetric affine types in characteristic zero, either classically or using quiver Hecke superalgebras \cite{kkt}, are likely to be possible.
\end{remark}

If $\a$ is a real root, let $L(\a)$ be the head of $\D(\a)$. If $\la$ is a multipartition, let $L(\la)$ be the head of $\D(\la)$. The modules $L(\la)$ are all simple and if $\la=\a^n$ then $L(\la)=L(\a)^{\circ n}$ (up to a grading shift which is not relevant to us).

Let $\la=(\a_1^{n_1},\cdots,\a_l^{n_l})$ be a root partition. Define the proper costandard module
\[
 \overline{\nabla}(\la)=L(\a_l^{n_l})\circ \cdots \circ L(\a_1^{n_1}).
\]

We can now state a standard classification result for simple modules of quiver Hecke algebras. Alternative references include \cite{kleshchev} and \cite{tingleywebster}.

\begin{theorem}\label{klrclass}
\cite[Theorem 8.8, Lemma 8.6]{mcn3}
 The simple modules for $R(\nu)$ are classified up to isomorphism and grading shift by root partitions. The simple module $L(\la)$ corresponding to the root partition $\la$ is the socle of $\overline{\nabla}(\lambda)$. Furthermore every other simple subquotient $L(\mu)$ of $\overline{\nabla}(\la)$ has $\la \prec \mu$ in the lexicographical order on root partitions.
\end{theorem}

\begin{proposition}\label{prop:extorthog}\cite[Proposition 24.3]{mcn3}
 Let $\D$ be a standard module and $\co$ be a proper standard module. Then for $i>0$,
 \[
  \Ext^i(\D,\co)=0.
 \]
\end{proposition}

If $M$ is a $R(\nu)$-module, then we define $\wt(M)=\nu$.

\begin{lemma}\label{ses}
 Let $\D$ be a root module for a root that is not simple. Then there are root modules $\D_\b$ and $\D_\ga$ and a nonzero $q$-integer $m$ such that there is a short exact sequence
\[
  0\to q^{-\b\cdot\ga} \D_\b \circ \D_\ga \xrightarrow{f_{\b\ga}} \D_\ga\circ\D_\b\to \D^{\oplus m}\to 0.
 \]
 If $\wt(\D)$ is not minimal in $\Phi^+\setminus \{\a_s\}$, then $\D_\b$ and $\D_\ga$ can be chosen to be in $\Phi^+\setminus \{\a_s\}$.
 Furthermore $f_{\b\ga}$ spans $\Hom(q^{-\b\cdot\ga} \D_\b \circ \D_\ga,\D_\ga\circ\D_\b)$ and $\HOM(q^{-\b\cdot\ga} \D_\b \circ \D_\ga, \D_\ga\circ\D_\b)$ is concentrated in nonnegative degrees.
\end{lemma}

\begin{proof}
In finite type, this is \cite[Theorem 4.10]{bkm}. In symmetric affine type,
this is \cite[Lemma 16.1]{mcn3} if $\wt(\D)$ is real and \cite[Theorem 17.1]{mcn3} if $\wt(\D)$ is imaginary.
 The statement about $\Hom(q^{-\b\cdot\ga} \D_\b \circ \D_\ga,\D_\ga\circ\D_\b)$ being one-dimensional is not explicitly mentioned in these references, but is clear from the proofs.
\end{proof}

\begin{definition}
 A module $M$ is said to have a $\D$-flag if there exists a filtration by submodules
 \[
  M=M_n\supset M_{n-1}\supset\cdots\supset M_1\supset M_0 =0
 \]
such that each subquotient $M_{i+1}/M_i$ is a standard module.
\end{definition}

\begin{lemma}\label{flag}\cite[Theorem 3.13]{bkm}
 A finitely generated module $M$ has a $\D$-flag if and only if $\Ext^1(M,\co)=0$ for all proper costandard modules $\co$.
\end{lemma}

\begin{lemma}\label{standardfpdim}
Every module which has a $\D$-flag has a finite projective resolution.
\end{lemma}

\begin{proof}
An induction on the length of the flag shows it suffices to prove this for standard modules. By Proposition \ref{roottostandard}, it suffices to prove this for root modules. 
By Lemma \ref{ses}, we reduce to the case of a root module for a simple root, which is projective.
\end{proof}

The last lemma means that for any module $M$ with a $\D$-flag, its class $[M]\in K_0(R\prmod)$ is defined.


\section{Reflection Functors}\label{sec:reflection}

Let $s\in S$. Let $_se\in R$ be element of the quiver Hecke algebra which is the sum of all generating idempotents with
first strand coloured $s$. Let $e_s$ be the sum of all generating idempotents with last strand coloured $s$. Let $_s\C$ (respectively $\C_s$) be the full subcategory of $R$-modules on which $_se$ (respectively $e_s$) acts by zero.

Equivalently
\[
 _s\C=R/\langle _se\rangle\mods \quad \mbox{and}\quad \C_s= R/\langle e_s\rangle\mods.
\]
%

%
%
%
%
%

\begin{lemma}\label{standardext}
 Suppose $P$ is projective in $_s\C$ (respectively in $\C_s$). Then $\Ext_R^i(P,M)=0$ for all $R$-modules $M$ in $_s\C$ (respectively $\C_s$) and $i>0$.
\end{lemma}

Note that the Ext group is computed in the category of $R$-modules rather than in $_s\C$ (or $\C_s$), so this result is not a tautology for $i>1$.

\begin{proof}
We consider the case where $P$ is projective in $_s\C$, the other case following similarly.
Choose a convex order $\prec$ which has $\a_s$ as its largest element (which always exists). All root partitions, standard and proper costandard modules appearing in this proof will be with respect to this convex order.
 First we will prove that $P$ has a $\D$-flag. By Lemma \ref{flag} it suffices to prove that $\Ext^1(P,\overline{\nabla})=0$ for all proper costandard modules $\overline{\nabla}$. 
 If $\la=(\a_s^{n_s},\a_2^{n_2},\a_3^{n_3},\ldots,\a_l^{n_l})$ is a root partition with $n_s=0$, then $\co(\la)\in {_s\C}$. So in this case $\Ext^1(P,\co(\la))=0$ as $P$ is projective in $_s\C$.
 If on the other hand $n_s\neq 0$, then $\co(\la)\cong X\circ L_s$ for some $X$. Then by (\ref{otheradjunction}),
 \[
  \Ext^1(P,\co(\la))\cong \Ext^1(\Res_{\a_s,\nu-\a_s}P,L_s\otimes X).
 \]
 Since $P\in {_s}\C$, this restriction is zero, as required in order to prove that $P$ has a $\D$-flag.
 
By Proposition \ref{prop:extorthog}, the fact that $P$ has a $\D$-flag implies that $\Ext^i(P,\co)=0$ for all proper costandard modules $\co$ and $i>0$.
 We now turn our attention to the statement that $\Ext_R^i(P,M)=0$ for all $R$-modules $M$ in $_s\C$.
Without loss of generality we may assume $M$ is simple. 

Then $M$ injects into a proper costandard module $\co$. Let $Q$ be the quotient $\co/M$. By Theorem \ref{klrclass}, every simple subquotient $L$ of $Q$ satisfies $L\prec M$ in the lexicographical order on root partitions.

By induction on the partial preorder $\prec$, we may assume $\Ext^i(P,L)=0$ for all $L\prec M$ and $i>0$. Hence $\Ext^i(P,Q)=0$ for $i>0$. Now apply
$\Hom(P,-)$ to the short exact sequence 
 \[
  0\to M\to \co \to Q\to 0.
 \]
 The resulting long exact sequence implies that $\Ext^i(P,M)=0$ for $i\geq 2$.
 
 This leaves only the $i=1$ case, but $\Ext^1(P,M)$ vanishes since $P$ is projective in $_s\C$ and $M$ is in $_s\C$.
\end{proof}
 %

For each $s\in S$, we have inclusions of full subcategories
\[
 _s\C,\C_s\subset R\mods \subset K^-(R\prmod).
\]
If our Cartan datum is of finite type then
$R$ has finite global dimension by \cite[Theorem 4.7]{klr1}. So in this case, the essential image lies in the bounded homotopy category $K^b(R\prmod)$.
In general, by Lemma \ref{standardfpdim}, the standard modules have finite projective dimension so lie in $K^b(R\prmod)$, while $R$ in general may have infinite global dimension.

If $P$ is a projective object in $\C_s$, then it is shown in the proof of Lemma \ref{standardext} above that $P$ has a $\D$-flag. By Lemma \ref{standardfpdim}, it therefore has a finite length projective resolution. Therefore we obtain a homomorphism of Grothendieck groups
\[
\chi_s:[\C_s]\to \f,
\]
where we use $[-]$ to denote the split Grothendieck group of projective objects, and recalling the isomorphism (\ref{klmain}). There is similarly $ {_s}\chi:[ {_s}\C]\to \f$.

\begin{theorem}
The homomorphisms $\chi_s$ and ${_s}\chi$ are injective with image $\ker(r_s)$ and $\ker({_s}r)$ respectively.
\end{theorem}

\begin{proof}
First we prove injectivity. For $i\in I$, let $L_i$ denote the corresponding simple and $P_i$ be its projective cover. Suppose that $\sum_{i\in I}a_i[P_i] \in \ker (\chi_s)$. Let $M$ be a finite dimensional $R$-module. We can apply $\dim\operatorname{RHom}(-,M)$ to obtain
\[
\sum_{i\in I} \sum_{j=0}^\infty (-1)^j \dim \Ext^j (P_j,M)=0.
\]
Fix $i\in I$ and set $M=L_i$. Then by Lemma \ref{standardext}, all terms vanish except the one involving $\Ext^0(P_i,L_i)$. Hence $a_i=0$ for all $i$. Therefore $\chi$ is injective.

We now show that $\im(\chi_s)\subset \ker(r_s)$.
If $M\in \C_s$, then $e_s M=0$, so $\Res_{\nu-s,s}(M)=0$. 
The results of \cite{khovanovlauda} imply that
\[
[\Res_{\nu-s,s} M] = r_s([M])\otimes \th_s,
\]
from which we obtain that $r_s([M])=0$, so $\im(\chi_s)\subset \ker(r_s)$.

It remains to show that $\ker(r_s)\subset \im(\chi_s)$. Suppose $\xi\in \ker(r_s)$. For each $i\in I$, let $a_i=\langle \xi,[L_i]\rangle$ and let $P=\sum_{i\in I} a_i[P_i]$. Then from the discussion above, using Lemma \ref{standardext}, 
\[
\langle \chi_s(P)-\xi,L_i\rangle =0
\]
for all $i\in I$.

Now let $L$ be a simple $R$-module that is not in $\C_s$. From the classification of simple $R$-modules, Theorem \ref{klrclass}, $[L]=y\th^\ast_s$ for some $y\in \f^\ast$. Therefore
\[
\langle \chi_s(P)-\xi,[L]\rangle = \langle \chi_s(P)-\xi,y\th_s^\ast\rangle=\langle r_s(\chi_s(P)-\xi),y\rangle.
\]
Since we've already shown that $\chi_s(P)\in \ker(r_s)$ and $\xi$ is chosen to be in $\ker(r_s)$, this pairing is zero. Since the Hom-pairing between $\f$ and $\f^\ast$ is non-degenerate, we have therefore shown that $\chi_s(P)-\xi=0$, proving that $\xi\in \im(\chi_s)=0$, as required.
\end{proof}

The above theorem establishes a sense in which ${_s}\C$ and $\C_s$ categorify $\ker({_s}r)$ and $\ker(r_s)$.

Our main result is the following theorem, which we prove at the end of this section.

\begin{theorem}\label{reflectionfunctor}
 Suppose our quiver Hecke algebra is simply laced, of finite or affine type. If we are in affine type, assume furthermore that the ground field is of characteristic zero. Then there is a monoidal equivalence of categories $_s\C\cong \C_s$ which decategorifies to Lusztig's braid group automorphism.
\end{theorem}

The restriction to simply laced is due to the generality of \cite{allr}. The further restrictions are needed since we rely on the theory of standard modules, developed in \cite{bkm} (building on \cite{klr1}) in finite type and in \cite{mcn3} in affine type over a field of characteristic zero.

In finite type ADE and in characteristic zero, Kato \cite{kato} has proved this equivalence geometrically. 
In finite type ADE and in all characteristics there is a geometric proof in \cite{geometry} using results of Maksimau \cite{even}.
The paper of Xiao and Zhao \cite{xiaozhao} provides an interpretation of $T_s$ in terms of perverse sheaves in all symmetric types, again in characteristic zero, which is subsequently generalised to all symmetrisable types in \cite{zhao}. Kato \cite{kato2} has shown how to use this to define $\T_s$ geometrically in characteristic zero in all symmetric types. A geometric approach to the monoidality of this functor, restricted again to characteristic zero, is in \cite{kato2,mc:monoidal}.
%
%


Choose a convex order ${_s}\!\!\prec$ which has $\a_s$ as its largest element (which always exists). From ${_s}\!\!\prec$ we can obtain
another convex order $\prec_s$, defined by
\begin{itemize}
 \item $\a_s\prec_s \b$  for all  $\b\in \Phi^+\setminus \{\a_s\}$
 \item $\b \prec_s\ga$ if $s\b\,{_s}\!\!\prec s\ga$ for all $\b,\ga\in\Phi^+\setminus \{\a_s\}$
\end{itemize}
These convex orders induce two families of standard modules in $R$-mod, we denote them by $_s\D(\la)$ and $\D_s(\la)$.

\begin{proposition}\label{tofstandard}
Let $\mu$ be a root partition whose $\a_s$ component is zero, and let $\la\in P$. Then
 \[
  \TT_s(\iota_\la ({_s}\D(\mu)))\cong \iota_{s\la}(\D_s(s\mu))
 \]
\end{proposition}

\begin{proof}
Since $\T_s$ is monoidal and additive, by Proposition \ref{roottostandard} it suffices to prove this for root modules. 
We will proceed by induction on the height of the root. First consider a root module for a minimal root - we define these to be roots $\a\in \Phi^+\setminus \{\a_s\}$ that cannot be expressed in the form $\b+\ga$ with $\b,\ga\in \Phi^+\setminus \{\a_s\}$. The proposition in this case follows from 
Theorem \ref{aller}.

Now suppose that $\a\in\Phi^+\setminus \{\a_s\}$ is not minimal. Consider the short exact sequence from Lemma \ref{ses}.
\[
  0\to q^{-\b\cdot\ga} {_s}\D(\b) \circ {_s}\D(\ga) \xrightarrow{f_{\b\ga}} {_s}\D(\ga)\circ{_s}\D(\b)\to {_s}\D(\a)^{\oplus m}\to 0.
 \]
Here $\b$ and $\ga$ are other roots in $\Phi^+\setminus \{\a_s\}$ with $\a=\b+\ga$ (they are a \emph{minimal pair} in the parlance of earlier works on quiver Hecke algebras). 
This identifies ${_s}\D(\a)^{\oplus m}$ as the cone of a nonzero morphism $f_{\b\ga}$ from $q^{-\b\cdot\ga} {_s}\D(\b) \circ {_s}\D(\ga)$ to ${_s}\D(\ga)\circ{_s}\D(\b)$. Recall from Lemma \ref{ses} that $\Hom_R(q^{-\b\cdot\ga} {_s}\D(\b) \circ {_s}\D(\ga),{_s}\D(\ga)\circ{_s}\D(\b))\cong k$.
By Theorem \ref{embedding}, $\iota_\la(f_{\b\ga})$ spans the one-dimensional space
\[
 \Hom_{K(\udot )} ( i_\la(q^{-\b\cdot\ga} {_s}\D(\b) \circ {_s}\D(\ga)),i_\la({_s}\D(\ga)\circ{_s}\D(\b)).
\]

Now apply $\T_s$. By inductive hypothesis, $\T_s(\iota_\la( {_s}\D(\b))\cong\iota_{s\la}(\D_s(s\b))$ and $\T_s(\iota_\la( {_s}\D(\ga))\cong\iota_{s\la}(\D_s(s\ga))$. Since $\T_s$ is an equivalence, $\T_s(\iota_\la(f_{\b\ga}))$ spans the one-dimensional space
\begin{equation}\label{homspace}
\Hom_{K(\udot)} (\iota_{s\la} (q^{\beta\cdot\ga} \D_s(s\b)\circ \D_s(s\ga)),\iota_{s\la}(\D_s(s\ga)\circ \D_s(s\b))).
\end{equation}
By Lemma \ref{ses} again, there is a short exact sequence
\[
0\to q^{-\b\cdot\ga} \D_s(s\b)\circ \D_s(s\ga) \xrightarrow{f'} \D_s(s\ga)\circ \D_s(s\b)\to \D(s\a)^{\oplus m} \to 0,
\]
and a similar argument shows that $\iota_{s\la}(f')$ also spans the one dimensional space in (\ref{homspace}). Therefore $\T_s(\iota_\la(f_{\b\ga}))$ and $\iota_{s\la}(f')$ are nonzero multiples of each other, hence their cones are isomorphic. This proves that $\T_s(\iota_\la( {_s}\D(\a))^{\oplus m}\cong\iota_{s\la}(\D_s(s\a))^{\oplus m}$ and we can take direct summands to conclude that $\T_s(\iota_\la( {_s}\D(\a))\cong\iota_{s\la}(\D_s(s\a))$, as required.
%
%
%
\end{proof}

\begin{lemma}\label{8.5}
Identify $_s\C$ and $\C_s$ with their essential images under $i_\la$ and $i_{s\la}$.
 Let $M$ be a module in $_s\C$ with a $\D$-flag. Then $\T_s(M)$ lies in $\C_s$ and has a $\D$-flag.
\end{lemma}

\begin{proof}
 We proceed by induction on the length of the $\D$-flag of $M$. When this length is 1, this is Proposition \ref{tofstandard}.
 So now suppose that $M$ has a $\D$-flag of length greater than one and this result is known for all modules with a smaller $\D$-flag than that of $M$.
 
 Then there is a short exact sequence
 \[
  0\to M'\to M \to M'' \to 0
 \]
 where $M'$ and $M''$ have a $\D$-flag of smaller length than that of $M$.

 The module $M$ is identified with the cone of a morphism from $M''[-1]$ to $M'$. Since $\T_s$ is exact, $\T_s(M)$ is the cone of a morphism from $\T_s(M'')[-1]$ to $\T_s(M)$. By induction on the length of a $\D$-flag, we know that $\T_s(M'')$ and $\T_s(M)$ are in $\C_s$. Therefore they are identified with a complex of projective $R$-modules, hence the same is true of $\T_s(M)$ since it appears as a cone. Now take the long exact sequence in homology associated to the triangle
 \[
  \T_s(M'')\to \T_s(M)\to \T_s(M')\xrightarrow{+1}.
 \]
 Since the homologies of $\T_s(M'')$ and $\T_s(M')$ are known to be concentrated in degree zero by our inductive hypothesis, the same is true of $\T_s(M)$. Hence $\T_s(M)$ is the class of a module, and appears as an extension of two modules with $\D$-flags. Therefore it lies in $\C_s$ and has a $\D$-flag.
\end{proof}

Consider a projective $P$ in $_s\C$.
In the proof of Lemma \ref{standardext}, we have shown that $P$ has a $\D$-flag. Thus by Lemma \ref{8.5}, $\T_s(P)$ lies in $\C_s$.

\begin{lemma}\label{extzero}
 Let $P$ be projective in $_s\C$ and let $\D$ be a standard module in $\C_s$. Then $\Ext^i(\T_s(P),\Delta)=0$ for all $i>0$.
\end{lemma}

\begin{proof}
 There are isomorphisms
 \[
  \Ext^i(\T_s(P),\Delta)\otimes \B_\la\cong \Ext^i(P,\T_s\inv(\Delta))\otimes \B_{s\la}.
 \]
As $\T_s\inv$ sends standard modules to standard modules, this is zero by Lemma \ref{standardext}.
\end{proof}

\begin{lemma}\label{technicalext}
Let $Q$ be a finitely generated $R$-module and let $d$ be a non-negative integer. Suppose that $\Ext^d(Q,L)=0$ for all simple $R$-modules $L$. Then $\Ext^d(Q,M)=0$ for all finitely generated modules $M$.
\end{lemma}

\begin{proof}
Let $P^\bullet$ be a projective resolution of $Q$. Since $R$ is Noetherian, we can, and do, choose this resolution so that each term is finitely generated. For an integer $e$, let $M(e)$ be the submodule of $M$ generated by all $M_n$ with $n\geq e$. Since $R$ is Laurentian and $P^d$ is finitely generated, there are no chain maps from $P^\bullet$ to $M(e)$ for sufficiently large $e$. Therefore $\Ext^d(Q,M(e))=0$ for such $e$. Similarly, we can simultaneously ensure that $e$ is large enough that $\Ext^{d-1}(Q,M(e))=0$. This implies that $\Ext^d(Q,M)\cong \Ext^d(Q,M/M(e))$. As $M$ is finitely generated and $R$ is Laurentian, the module $M/M(e)$ is finite dimensional and a simple induction on the length of a Jordan-Holder series shows that $\Ext^d(Q,N)=0$ for any finite dimensional module $N$.
\end{proof}

\begin{lemma}\label{projtoproj}
Let $P$ be a projective object in $_s\C$. Then
 $\T_s(P)$ is projective in $\C_s$.
\end{lemma}

\begin{proof}
It suffices to show that $\Ext^1(\T_s(P),L)=0$ for every simple module $L$ in $\C_s$.

We will do this by showing that $\Ext^d(\T_s(P),L)=0$ for every simple module $L$ and every integer $d\geq 1$ by a decreasing induction on $d$.

Since $\T_s(P)$ has a $\D$-flag, it has finite projective dimension. Therefore this result is true for all sufficiently large $d$ which establishes the base case of the induction.

For our inductive hypothesis, let us now assume that $\Ext^{d+1}(\T_s(P),L)=0$ for every simple module $L$.

Let $L$ be a simple module in $\C_s$. 
 and consider the short exact sequence
\[
 0\to K \to \Delta \to L\to 0
\]
where $\D$ is the unique standard module surjecting onto $L$.

Apply $\Hom(\T_s(P),-)$ and consider the corresponding long exact sequence of Ext groups. By Lemma \ref{extzero}, we obtain an isomorphism $$\Ext^d(\T(P),L)\cong \Ext^{d+1}(\T(P),K).$$ 

Note that $K$ is finitely generated since $R$ is Noetherian. By Lemma \ref{technicalext} and the inductive hypothesis, $\Ext^{d+1}(\T(P),K)=0$. This completes the inductive step of the proof and hence completes the proof of the Lemma.
%
%
\end{proof}

\begin{proof}[Proof of Theorem \ref{reflectionfunctor}]
Identify $_s\C$ and $\C_s$ with their essential images under the faithful functors $i_\la$ and $i_{s\la}$.
Let $P$ be projective in $_s\C$. Then by Lemma \ref{projtoproj},
there is a projective $Q$ in $\C_s$ such that $\T_s(i_\la(P))\cong i_{s\la}(Q)$. Since $\T_s$ is an equivalence, by Theorem \ref{embedding}, there is an induced isomorphism
\[
 \End_R(P)\otimes \B_\la\cong \End_R(Q)\otimes \B_{s\la}.
\]
As $\T_s$ also induces an isomorphism $\B_\la\cong \B_{s\la}$, we get an induced isomorphism
\[
 \End_R(P)\cong \End_R(Q).
\]

By Morita theory, an abelian category is governed by the endomorphism algebra of a projective generator. Since $\T_s\inv$ induces an analogous isomorphism, it must be that $\T_s$ induces an equivalence of categories $_s\C\cong \C_s$, as required.

The fact that $\T_s$ decategorifies to Lusztig's braid group action follows from the fact it is monoidal and the identities (\ref{tsdef}) and (\ref{tsinv}).
\end{proof}

%
%
 \section{restriction of categorical representations}
 
We remind the reader that $\Phi$ is simply laced, and of finite or affine type.

\begin{definition}\label{facedefinition}
 A face is a decomposition of $\Phi^+$ into three disjoint subsets
 \[
  \Phi^+=F^+\sqcup F\sqcup F^-
 \]
such that, for all $x\in\spann_{\R_{\geq 0}}F$:
\begin{enumerate}
\item If $y\in \spann_{\R_{\geq 0}}F^+$ is non-zero, then  $x+y\notin \spann_{\R_{\geq 0}}(F^- \cup F)$.
 \item If $y\in \spann_{\R_{\geq 0}}F^-$ is non-zero, then  $x+y\notin \spann_{\R_{\geq 0}}(F^+ \cup F)$.
\end{enumerate}
\end{definition}

We often abuse notation and use $F$ to refer to the entire face.

\begin{example}\label{faceexample}
Suppose $\Phi^+$ is of type $A_n^{(1)}$ and let $1\leq e\leq n$ be an integer. Let $\la\map {\R\Phi^+}{\R}$ be a generic linear map subject to the conditions
\[
 \la(\a_0)=\la(\a_1)=\cdots=\la(\a_{e-1})=0=\la(\a_e+\a_{e+1}+\cdots+\a_n) 
\]and
$\la(\a_e)<\cdots < \la(\a_n)$.

Then $F=\la\inv(0)\cap \Phi^+$ is a face of type $A_e^{(1)}$. 
\end{example}
We choose this example because it appears in \cite{maksimauth} and \cite{richewilliamson}. Also it is not conjugate to a standard face under the Weyl group so the 2-functor of Theorem \ref{homotopyface} cannot be obtained as a composition of reflection functors $\T_s$.

Here is an alternative viewpoint on faces:

\begin{theorem}\cite[Lemma 1.10]{tingleywebster}\label{linearfunctionals}
 For any face $\Phi^+=F^+\sqcup F\sqcup F^-$, there is a sequence of linear functionals $\{\la_n\}_{n\in\N}$ on $\R\Phi$ such that
  \begin{enumerate}
   \item $F\subset \ker (\la_n)$ for all $n$, 
   \item For all $\a\in F^+$, $\la_n(\a)>0$ for all $n\gg0$,
   \item For all $\a\in F^-$, $\la_n(\a)<0$ for all $n\gg 0$.
  \end{enumerate}
\end{theorem}


Fix a face $F$ of $\Phi$. Let $\Delta_F$ be the set of positive real roots in $F$ which are not sums of other positive roots in $F$. Let $\Phi_F$ be the corresponding root system whose simple roots are $\D_F$. Since $\Phi$ is at worst of affine type, then as discussed in \cite[\S 3.2]{tingleywebster}, $\Phi_F$ is a product of finite and affine root systems (although imaginary root spaces may decompose, see \cite[Remark 3.16]{tingleywebster} for an example).

Under our standing assumptions on $\Phi$, namely that it is either $\Phi$ of finite type, or of symmetric affine type and the field $k$ is of characteristic zero, then to each face $F$, a quiver Hecke algebra $R_F$ is constructed in \cite{facefunctors}. 
This is a quiver Hecke algebra for the face root system $\Phi_F$. There is a choice of polynomials $Q_{ij}$ to define this face quiver Hecke algebra which is implicitly determined in \cite{facefunctors}. With the same choice of parameters, we can define the face 2-category $\U_F$. 


This construction is actually more general than as stated in the paragraph above, see \cite[Assumption 3.11]{facefunctors} for a precise statement. The restrictions on $\Phi$ and the characteristic of $k$ ensure that a theory of standard modules exists. However, this construction does not require the full theory of standard modules, but only for the theory of the root modules $\D(\a)$ for each $\a\in\D_F$. Such a theory of standard modules exists for the face of Example \ref{faceexample} in all characteristics, as mentioned in \cite[Assumption 3.11]{facefunctors}. Therefore we can meaningfully talk about $\U_F$ and its offshoots in this example.

We now define a 2-functor from $\U_F$ to $K^b(\udot)$. Since $\U_F$ is defined by generators and relations, to define this 2-functor, it suffices to give the image of each of the generators. We check the relations in the proof of Theorem \ref{homotopyface}.

The root lattice $P_F$ of $\Phi_F$ comes equipped with a natural map to $P$, which determines what our 2-functor does on objects.

For $\a\in \D_F$ and $\la\in P_F$, we send $\F_\a 1_\la$ to $i_\la(\D(\a))$. We send $\E_\a1_\la$ to its biadjoint, and $\eta$ and $\nu$ to the corresponding adjunction maps.

The image of the 2-morphism $\tau$ is the image of the corresponding element of $\Hom(\D(\a)\circ \D(\b),\D(\b)\circ\D(\a))$ determined in \cite[Lemma 3.10]{facefunctors}. The image of the 2-morphism $x$ is the image of the corresponding endomorphism of $\D(\a)$ determined by \cite[Theorem 3.12]{facefunctors}.
 
 \begin{theorem}\label{homotopyface}
  The above assignments define a 2-functor $\U_F \to K^b(\U)$.
 \end{theorem}

 \begin{proof}
Since $\U_F$ is given by a presentation, we have to check the defining relations hold. 

We first make an observation about faces. Let $(F^-,F,F^+)$ be a face. Suppose $s\in S$ is such that the corresponding simple root $\a_s$ satisfies $\a_s\in F^-$. Then we can define a new face
 \[
  \sigma_s(F^-,F,F^+)=(s(F^-)\setminus\{-\a_s\},s(F),s(F^+)\sqcup\{\a_s\}).
 \]
 Conversely if $\a_s\in F^+$ then there is a new face
 \[
  \sigma_s^*(F^-,F,F^+)=(s(F^-)\sqcup\{\a_s\},s(F),s(F^+)\setminus\{-\a_s\}).
 \]

 In terms of the sequence of linear functionals $\{\la_n\}_{n\in\N}$ from Theorem \ref{linearfunctionals}, both of these constructions arise from the sequence $\{s(\la_n)\}_{n\in\N}$. If $\a_s\in F$, then applying $s$ to each of the functionals in this sequence does not change the face.
 
Suppose $\a\in \Delta_F$. We can apply a sequence of transformations of the above type to transform $\a$ into an element of $S$. By Proposition \ref{tofstandard}, the corresponding composition of $\T_s$'s will send $\F_s$ to $\D(\a)$. We therefore conjugate by a composite of Rickard complexes, so the one-colour relations must be preserved. This shows that the one-colour relations all hold.

The validity of the quiver Hecke relations is checked in \cite[Theorem 3.23]{facefunctors}. 
%
%
 It remains to show that the image of a rightward crossing (\ref{rightcross}) with two different colours is invertible.
  
  If the face is of affine type, $\Phi$ must be simply laced of affine type, hence every root is in the $W$-orbit of $\a_0$, where $\a_0$ is the simple affine root and $W$ is the Weyl group. Let $\a\in \D_F$. Pick $w\in W$ such that $w\a=\a_0$. 
  Let $\{\la_n\}_{n\in\N}$ be the sequence of linear functionals associated to $F$ from Theorem \ref{linearfunctionals}.
  We apply $w$ to the sequence $\{\la_n\}_{n\in\N}$ to get a new face related by a sequence of reflection functors $\T_s$ and $\T_s\inv$. Therefore we may assume without loss of generality that any given root $\a\in \D_F$ is $\a_0$. So given two roots $\a,\b\in\D_0$, without loss of generality $\a=\a_0$. Then since $\a+\b$ is a positive summand of $\d$ and the coefficient of $\a_0$ in $\d$ is 1, there is no occurrence of $\a_0$ in $\b$. Thus a projective resolution of $\D(\b)$ by standard projective modules has no occurrence of $P_0$. So the map $\E_0\D(\b)\to \D(\b)\E_0$ as a map of complexes consists of all rightward crossings. It is then clear that the inverse of this map of complexes is the corresponding map with all leftward crossings.
  
  If the face is of finite type then either $F^-$ or $F^+$ is finite. Without loss of generality, assume $F^-$ is finite. As in the one-colour case, we apply a finite number of transformations $\sigma_s$ to reduce to the case when $F^-$ is the empty set. Now there is a functional $\la$ such that $\la(\a)\geq 0$ for all $\a\in\Phi^+$, and $\ker(\la)\cap \Phi^+ = F$. Then it is clear that $F$ is a standard face, i.e. arising from an inclusion $J\subset I$, and in this case the result is obvious.
 \end{proof}

 It is natural to make the following conjecture. It is not obvious from the results discussed here since $K^b(K^b(\mathcal{A}))\not\cong K^b(\mathcal{A})$ for a general additive category $\A$.
 
 \begin{conjecture}
  This 2-functor from $\U_F$ to $K^b(\U)$ extends to a faithful 2-functor from $K^b(\U_F)$ to $K^b(\U)$.
 \end{conjecture}

 \begin{remark}
  In finite type the face 2-functors $K^b(\udot_F)\to K^b(\udot)$ are all compositions of the equivalences $\T_s$ and $\T_s\inv$, together with the inclusions of a face obtained from a subset of $S$, and so in particular are faithful (the inclusions of these latter types of faces are faithful by the nondegeneracy theorem).
 \end{remark}

 
 As a consequence, any time $\U$ acts on a category $\A$, we can restrict along the face to obtain an action of $\U_F$ on $K^b(\A)$. In special cases (such as those considered in \cite{maksimauth} and \cite{richewilliamson}), we actually get an action of $\U_F$ on $\A$. It would be interesting to have an elegant and practical criterion to determine when this restricted action is an action on $\A$ instead of only an action on $K^b(\A)$. In lieu of a beautiful criterion, we now state a necessary condition which is enough to reconstruct the categorical restrictions of \cite{maksimauth,richewilliamson}.
 
 For all $\a\in \D_F$, choose a projective resolution of $\D(\a)$. For each indecomposable projective appearing in a resolution except for those in homological degree zero, choose an inclusion $P\subset P_\ii$ as a direct summand. Write $\ii=(i_1,\ldots,i_n)$. Then our criterion is, that for all weights $\la$ and for all such $\ii$, that at least one of the categories
 \[
  \A(\la),\A(\la+i_1),\ldots, \A(\la+i_1+\cdots+i_n)
 \]
 is zero.
 
 This criterion works because it implies that for each generating 1-morphism of $K(\U_F)$, there is an endofunctor of $\A$ inducing the same action on $K(\A)$.
 
 This condition can be checked at the level of Grothendieck groups.
 For example, there is a categorical action of $\hat{\mathfrak{sl}}_p$ on the principal block of $\Rep(GL_n;\overline{\mathbb{F}}_p)$ for $n\geq p$, constructed in \cite{richewilliamson}. At the level of Grothendieck groups, this categorifies the $n$-th exterior power of the natural representation of $\hat{\mathfrak{sl}}_p$ on $\mathbb{C}^p\otimes \mathbb{C}[t,t\inv]$.
 
%
%
%
%
 
\bibliographystyle{alpha}
\def\cprime{$'$}

\end{document}